\documentclass[a4paper,11pt]{amsart}

\usepackage{amsmath,amssymb,amsfonts,latexsym}

\newtheorem{thm}{Theorem}[section]

\newtheorem{lem}[thm]{Lemma}
\newtheorem{prop}[thm]{Proposition}
\newtheorem{defn}[thm]{Definition}
\newtheorem{rem}[thm]{\bf{Remark}}

\newtheorem{Pro}[thm]{\bf{Problem}}

\begin{document}

\title[On the isomorphism problem for central extensions I]{On the isomorphism problem for central extensions I}

\author{Noureddine Snanou}
\address{Department of Mathematics, Faculty of Sciences Dhar El Mahraz, Sidi Mohamed Ben Abdellah University, Fez, Morocco}
\email{noureddine.snanou@usmba.ac.ma}


\begin{abstract}
Let $G_{2}$ be a group which acts trivially on an abelian group $G_{1}$. As is well known, each perturbed direct product of $G_{1}$ and $G_{2}$ under a 2-cocycle $\varepsilon\in Z^{2}(G_{2},G_{1})$ determines a central extension of $G_{1}$ by $G_{2}$. The purpose of this paper is to study perturbed direct products of groups and to decide in some cases how 
the isomorphism of these groups can be decided. Furthermore, we show that the study of the isomorphism of perturbed direct products of an abelian torsion group and a finite group is 
reduced to the study of the isomorphism of $p$-subgroups. We characterize such isomorphisms in various situations with some assumptions on the quotient group.

\vspace{2mm}

\vspace{2mm}


\end{abstract}


\maketitle


\section{Introduction}


Deciding the isomorphism of two given groups or even classifying all groups in a certain class is one of the most classical and challenging problems in group theory.
The classification of finite simple groups is the first step of the H\"{o}lder program which gives us a complete list of finite simple groups \cite{As04}. A group $G$
that is not simple can be broken into two smaller groups, namely a nontrivial normal subgroup $G_{1}$ (the kernel group) and the corresponding quotient group $G_{2}\cong G/G_{1}$.
This is equivalent to say that $G$ is an extension of $G_{1}$ by $G_{2}$. In particular, if $G_{1}$ is a central subgroup of $G$, then we say that $G$ is a central extension of $G_{1}$
by $G_{2}$. The question of what groups $G$ are extensions of $G_{1}$ by $G_{2}$ is called the extension problem and this is the second step of the H\"{o}lder program. The solution
 to the extension problem would give us a complete classification of all finite groups. But, it is not easy to solve this problem, and no general theory exists which characterizes
 all possible extensions at one time. However, for group extensions with abelian kernel, an answer to the extension problem has been given by H\"{o}lder and Schreier by using the
 group cohomology, but it has some considerable disadvantages \cite[Theorem 7.34]{ROT95}. In fact, this answer does not allow us to
compute the number of non-isomorphic extensions of $G_{1}$ by $G_{2}$ (the isomorphism problem). Very recently, we study in \cite{SCC22} the isomorphism problem for split extensions.
In \cite{S-C20, Sn20}, we characterize the isomorphism problem for non-split abelian extensions. In fact, most of the results of those studies do not concern general isomorphisms,
but only those of certain type, namely leaving one of the two factors or even both invariant. The aim of this paper is to give a further contribution to this topic. More precisely, we complete the work with the isomorphism problem for central extension in other special cases. We mainly deal with isomorphisms inducing the identity or a commuting automorphism on the quotient group. Further, we show that the study of the isomorphism of central extensions of an abelian torsion group by a finite group is reduced to the study of the isomorphism of $p$-subgroups. In this direction, we characterize such isomorphisms in various situations with some assumptions on the quotient group.

Let $G$ be a group. As usual we denote by $Z(G)$, $G'$ and $Aut(G)$, respectively, the center, the derived subgroup, and the automorphism group of $G$. If  $G$ is finite, then $\pi(G)$ denotes the set of prime divisors of the order of $G$. 


\section{Preliminaries and Properties}


Let $G_{2}$ be a group which acts trivially on an abelian group $G_{1}$. A normalized 2-cocycle of $G_{2}$ with coefficients in $G_{1}$ is a map $%
\varepsilon :G_{2}\times G_{2}\rightarrow G_{1}$ satisfying the following two conditions:

\begin{eqnarray}\label{normlized}
\varepsilon(g,1)&=&\varepsilon(1,g)=1 \text{ \ for all \ } g\in G_{2}.\\
\label{Cocycle}
\varepsilon(h,g)\varepsilon (hg,k)&=&\varepsilon
(g,k)\varepsilon (h,gk) \text{ \ for all \ } g, h, k\in G_{2}.
\end{eqnarray}
The condition given by the equation \eqref{normlized} is called the normalization condition, and the condition given by \eqref{Cocycle} is referred to as the 2-cocycle condition. The set of normalized 2-cocycles of $G_{2}$ with coefficients in $G_{1}$ is an abelian group and denoted by $Z^{2}(G_{2},G_{1})$. The trivial 2-cocycle is the 2-cocycle $c$ with $c(g,h)=1$ for all $g,h\in G_{2}$. Note that the elements of $Z^{2}(G_{2},G_{1})$ are known by factor sets in many books, (see for example \cite{Bro82, Mac63, ROT95, W94}). The set of all normalized 2-cocycles which are symmetric forms a
subgroup of $Z^{2}(G_{2},G_{1})$ and denoted by $SZ^{2}(G_{2},G_{1})$. A $2$-coboundary of $G_{2}$ with coefficients in $G_{1}$ is a map $\psi :G_{2}\times G_{2}\rightarrow G_{1}$ satisfying that
for all $y,y'\in G_{2}:\psi (y,y')=\eta( y) \eta ( yy') ^{-1}\eta( y') $ for some $\eta
:G_{2}\rightarrow G_{1}$. The set of 2-coboundaries of $G_{2}$ with coefficients in $G_{1}$ is a subgroup of $ Z^{2}(G_{2},G_{1})$ and denoted by $B^{2}( G_{2},G_{1})$. The corresponding factor group $H^{2}(G_{2},G_{1})=Z^{2}(G_{2},G_{1})/B^{2}(G_{2},G_{1})$ is called
the second cohomology group of $G_{2}$ with coefficients in $G_{1}$. The elements of
$H^{2}(G_{2},G_{1})$ are called cohomology classes. The cohomology class of $\varepsilon\in Z^{2}(G_{2},G_{1})$ is denoted by $[\varepsilon]$. Two normalized 2-cocycles are said to be cohomologous if they lie in the same cohomology class.

Let $1\rightarrow G_{1}\overset{i}{\rightarrow }G\overset%
{j}{\rightarrow }G_{2}\rightarrow 1$ be a short exact sequence of groups, i.e., an extension of a group $G_{1}$ by the group $G_{2}\cong G/G_{1}$. If $G_{1}$ is a central subgroup of $G$, then such an extension is called a central extension of $G_{1}$ by $G_{2}$. We refer to $G_{1}$ as the kernel group, and $G_{2}$ as the quotient group for the extension. The Schreier's theorem says that the central extensions of $G_{1}$ by $G_{2}$ are classified by the
non-trivial elements of the second cohomology group $H^{2}(G_{2},G_{1})$ with coefficients in $G_{1}$ \cite[Theorem 7.34]{ROT95}. Split extensions correspond to the trivial equivalence class of $H^{2}(G_{2},G_{1})$.

Let $G_{2}$ be a group which acts trivially on an abelian group $G_{1}$. It is well known that each perturbed direct product of $G_{1}$ and $G_{2}$ under a 2-cocycle $\varepsilon$ determines a central extension of $G_{1}$ by $G_{2}$. Let $\varepsilon \in Z^{2}(G_{2},G_{1})$, the perturbed direct product of $G_{1}$ and $G_{2}$ under $\varepsilon$ is defined as the group $G_{1}\underset{\varepsilon }{\times }G_{2}$ with underlying set $G_{1}\times
G_{2}$ and operation given by
\begin{equation*}
(x,y) \underset{\varepsilon}{\cdot }(x',y')=(xx'\varepsilon(y,y'),yy')
\end{equation*}
for all $x$, $x'\in G_{1}$ and $y$, $y'\in G_{2}$. The converse is also true, then each central extension $G$ of $G_{1}$ by $G_{2}$ is isomorphic to a perturbed direct product of $G_{1}$ and $G_{2}$, namely $G\cong G_{1}\underset{\varepsilon }{\times }G_{2}$ for some $\varepsilon \in Z^{2}(G_{2},G_{1})$  \cite[Proposition 2.3]{S-C20}. We can easily see that the perturbed direct product $G_{1}%
\underset{\varepsilon}{\times }G_{2}$ is abelian if and only if $%
G_{2}$ is abelian and $\varepsilon
\in SZ^{2}(G_{2},G_{1})$. In particular, we have $G_{1}\underset{\varepsilon}{\times }G_{2}= G_{1}\times G_{2}$ if and only if $\varepsilon$ is the trivial 2-cocycle. But, it is possible for a direct product to be isomorphic to a perturbed direct product as Remark \ref{Product} shows. In particular, suppose that $G_{1}$ and $G_{2}$  are abelian groups and $\varepsilon$ is non-symmetric. So $G_{1}\underset{\varepsilon}{\times }G_{2}$ is non-abelian and then $G_{1}\underset{\varepsilon}{\times }G_{2}\ncong G_{1}\times G_{2}$. Furthermore, we have

\begin{prop}
Let $G_{1}$ be a finite abelian group and $G_{2}$ a finite group and let $\varepsilon \in
Z^{2}(G_{2},G_{1})$. If $SZ^{2}(G_{2},G_{1})=\{1\}$, then $G_{1}\underset{\varepsilon}{\times }G_{2}\cong G_{1}\times G_{2}$ if and only if $\varepsilon=1$.
\end{prop}

\begin{proof}
The if direction is clear. Conversely, suppose that $\varepsilon\neq 1 $, by using the 2-cocycle condition, we get $[(x,y),(x',y')]=(\varepsilon(y,y')\varepsilon(y',y)^{-1}, [y,y'])$ for all $(x,y)$, $(x',y')\in G_{1}\underset{\varepsilon}{\times }G_{2}$. So, $(G_{1}\underset{\varepsilon}{\times }G_{2})'\cong H_{\varepsilon}\underset{\varepsilon'}{\times }G'_{2}$ such that $\varepsilon'=res_{G_{2}'\times G_{2}'}(\varepsilon) $ and $H_{\varepsilon}$ is generated by the elements of the form $\varepsilon(y,y')\varepsilon(y,y')^{-1}$ where $y$, $y'\in G_{2}$. By assumption, we have $H_{\varepsilon}$ is a nontrivial subgroup of $G_{1}$. But, $(G_{1}\times G_{2})'=G_{2}'$ which implies that $G_{1}\underset{\varepsilon}{\times }G_{2}\ncong G_{1}\times G_{2}$, as required.
\end{proof}

\begin{rem}\label{Product}Let $\varepsilon\in Z^{2}(G_{2},G_{1})$ and $G=G_{1}\underset{\varepsilon}{\times }G_{2}$ be a finite group. Under some conditions on $G_{1}$ and $G_{2}$, the group $G$ can also be decomposed as a direct product of $G_{1}$ and $G_{2}$. Indeed, by \cite[Proposition 2.1.7]{Kar87}, $G'\cap G_{1}$ is isomorphic to a subgroup of the Schur multiplier $M(G/G_{1})$ of $G/G_{1}$. So if $|G_{1}|$ and $|M(G_{2})|$ are coprime, then $G'\cap G_{1}=1$. Further, if $G_{2}$ is perfect, then so is $G/G_{1}$. Hence $G/G_{1}= G'G_{1}/G_{1}$, which implies that $G= G'G_{1}$. Thus, $G= G'\times G_{1}$ and then $G\cong G_{1}\times G_{2}$.
\end{rem}

Now, in view of the preceding discussion, the following problem seems natural.

\begin{Pro}
Find necessary and sufficient conditions on $\varepsilon_{1}$ and $\varepsilon_{2}$ under which the central extensions $G_{1}\underset%
{\varepsilon_{1}}{\times }G_{2}$ and $G_{1}\underset%
{\varepsilon_{2}}{\times }G_{2}$ are isomorphic.
\end{Pro}

To begin, let $\varepsilon_{1}$, $\varepsilon_{2}\in Z^{2}(G_{2},G_{1})$ and $\varphi$ a group homomorphism from $G_{1}\underset%
{\varepsilon_{1}}{\times }G_{2}$ to $G_{1}\underset{\varepsilon_{2}}{\times }G_{2}$. Let $pr_{i}:G_{1}\underset{\varepsilon_{2}}{\times }G_{2}\rightarrow G_{i}$ be the $%
ith$ canonical projection and $t_{i}:G_{i}\rightarrow G_{1}\underset{\varepsilon_{1}}{\times }G_{2}$ be the $ith$ canonical injection. Set $\varphi
_{ij}=pr_{i}\circ \varphi \circ t_{j}$, where $1\leq i,j\leq 2$. So we can write $\varphi $ in the matrix form: $\varphi =\left(
\begin{array}{cc}
\varphi _{11} & \varphi _{12} \\
\varphi _{21} & \varphi _{22}%
\end{array}%
\right) $. Furthermore, we have the following lemma which we need in the sequel.

\begin{lem}\cite[Lemma 3.1]{S-C20}
\label{class1} Let $\varphi=\left(
\begin{array}{cc}
\varphi _{11} & \varphi _{12} \\
\varphi _{21} & \varphi _{22}%
\end{array}%
\right) $ be a group homomorphism from $G_{1}\underset%
{\varepsilon_{1}}{\times }G_{2}$ to $G_{1}\underset{\varepsilon_{2}}{\times }G_{2}$. Then
\begin{eqnarray}\label{Hom}
\varphi(x,y)=(\varphi_{11}(x)\varphi_{12}( y) \varepsilon _{2}( \varphi
_{21}( x),\varphi_{22}( y)),\text{ }\varphi _{21}(x) \varphi_{22}( y))
\end{eqnarray}
for all $x\in G_{1}$, and $y\in G_{2}$.
\end{lem}


\section{Isomorphisms inducing the identity on the quotient group}


\begin{defn}
The perturbed direct products $G_{1}\underset{\varepsilon _{1}}{\times }G_{2}$ and $%
G_{1}\underset{\varepsilon _{2}}{\times }G_{2}$ are called
$G_{2}$-isomorphic if there exists an isomorphism $\varphi=\left(
\begin{array}{cc}
\varphi _{11} & \varphi _{12} \\
\varphi _{21} & \varphi _{22}%
\end{array}%
\right) $  between them such that $\varphi_{22}=id_{G_{2}}$.
\end{defn}

In the following, we give an interesting result for a special class of non-nilpotent quotient groups, namely for those that have trivial center.

\begin{prop}\label{centerless}
   Let $G_{2}$ be a centerless group which acts trivially on an abelian group $G_{1}$. The perturbed direct products $G_{1}\underset{\varepsilon_{1}}{\times }G_{2}$
   and $G_{1}\underset{\varepsilon_{2}}{\times }G_{2}$ are $G_{2}$-isomorphic if and only if there exists $\sigma \in Aut(G_{1})$ such that
$(\sigma\circ\varepsilon_{1})\varepsilon_{2}^{-1} \in B^{2}( G_{2},G_{1})$.
\end{prop}

\begin{proof}
 Suppose that the perturbed direct products $G_{1}\underset{\varepsilon_{1}}{\times }G_{2}$ and $G_{1}\underset{\varepsilon_{2}}{\times }G_{2}$ are isomorphic by an isomorphism
 $\varphi=\left(
     \begin{array}{cc}
     \varphi _{11} & \varphi _{12} \\
     \varphi _{21} & id_{G_{2}}%
     \end{array}%
     \right) $. Since $pr_{2}$ and $t_{1}$ are group homomorphisms, so is $\varphi _{21}$. Furthermore, we see that $\varphi (x,1) \underset{\varepsilon
_{2}}{\bullet }\varphi ( 1,y)=\varphi (1,y) \underset{\varepsilon
_{2}}{\bullet }\varphi ( x,1)$. So by applying formula \eqref{Hom}, we get $\varphi _{21}(x)y=y\varphi _{21}(x)$ for all $x\in G_{1}$, and $y\in G_{2}$. Thus, we have $\varphi _{21}\in Hom(G_{1},Z(G_{2}))$. Since $G_{2}$ is centerless, it follows that $\varphi _{21}=1$. Hence, by \cite[Theorem 3.7]{Sn20}, there exists $\sigma=\varphi _{11} \in Aut(G_{1})$ such that $(\sigma\circ\varepsilon_{1})\varepsilon_{2}^{-1} \in B^{2}( G_{2},G_{1})$. Conversely, since $(\sigma\circ\varepsilon_{1})\varepsilon_{2}^{-1} \in B^{2}( G_{2},G_{1})$, it follows that there exists a map $\eta:G_{2}\rightarrow G_{1}$ such that $((\sigma\circ\varepsilon_{1})\varepsilon_{2}^{-1})(y,y')=\eta(y)\eta(y')\eta(yy')^{-1}$ for all $y$, $y'\in G_{2}$. By the normalization condition, we have $\eta(1)=1$. So, the bijection $\varphi$ defined by $\varphi(x,y)=(\sigma(x)\eta(y),\text{ }y)$ is clearly an isomorphism. As required.
\end{proof}

Let $G_{1}$ be an abelian torsion group, i.e. all elements of $G_{1}$ are of finite order. Then $G_{1}$ is a restricted direct product of all $p$-components $G_{1p}$, where $p$ runs through the set of prime numbers. Let $G_{2}$ be a finite group which acts trivially on $G_{1}$ and $\varepsilon\in Z^{2}(G_{2},G_{1})$. Let $\pi(G_{2})=\{p_{1}, p_{2}, \ldots, p_{k}\}$ and $G_{2i}$ be a Sylow $p_{i}$-subgroup of $G_{2}$ for each $1\leq i\leq k$. Clearly, we have $\varepsilon_{i}=res_{G_{2i}}(\varepsilon)\in Z^{2}(G_{2i},G_{1})$. So, in view of the previous preposition, we get the following result.

\begin{thm}\label{centerthm}
Keep the preceding notations and assumptions and suppose that $G_{2}$ is centerless. The perturbed direct products $G_{1}\underset{\varepsilon_{1}}{\times }G_{2}$ and $G_{1}\underset{\varepsilon_{2}}{\times }G_{2}$ are $G_{2}$-isomorphic
if and only if $G_{1p_{i}}\underset{\varepsilon_{1i}}{\times }G_{2i}$ and $G_{1p_{i}}\underset{\varepsilon_{2i}}{\times }G_{2i}$ are $G_{2i}$-isomorphic for all $1\leq i\leq k$.
\end{thm}

\begin{proof}Suppose that the groups $G_{1}\underset{\varepsilon_{1}}{\times }G_{2}$ and $G_{1}\underset{\varepsilon_{2}}{\times }G_{2}$ are $G_{2}$-isomorphic. By  Proposition \ref{centerless}, there exists $\sigma \in Aut(G_{1})$ such that $(\sigma\circ\varepsilon_{1})\varepsilon_{2}^{-1} \in B^{2}( G_{2},G_{1})$. So $res_{G_{2i}}((\sigma\circ\varepsilon_{1})\varepsilon_{2}^{-1})\in B^{2}(G_{2i},G_{1})$. Since $G_{2i}$ is a $p_{i}$-group and $G_{1}$ is torsion, by \cite[Lemma 1.5]{K93}, it follows that $H^{2}(G_{2i},G_{1})=H^{2}(G_{2i},G_{1p_{i}})$. Further, we have $\sigma_{i}=res_{G_{1p_{i}}}(\varphi_{11})\in Aut(G_{1p_{i}})$. So $(\sigma_{i}\circ\varepsilon_{1i})\varepsilon_{2i}^{-1}\in B^{2}(G_{2i},G_{1p_{i}})$, which implies necessity. Conversely, suppose that $G_{1p_{i}}\underset{\varepsilon_{1i}}{\times }G_{2i}$ and $G_{1p_{i}}\underset{\varepsilon_{2i}}{\times }G_{2i}$ are $G_{2i}$-isomorphic for all $1\leq i\leq k$. So by Proposition \ref{centerless}, there exists $\sigma_{i}\in Aut(G_{1p_{i}})$ such that $(\sigma_{i}\circ\varepsilon_{1i})\varepsilon_{2i}^{-1}\in B^{2}(G_{2i},G_{1p_{i}})$ for each $1\leq i\leq k$. So $\sigma=(\sigma_{i})_{1\leq i\leq k}\in Aut(G_{1})$ and then $res_{G_{2i}}([(\sigma\circ\varepsilon_{1})\varepsilon_{2}^{-1}])=1$ in $H^{2}(G_{2i},G_{1p_{i}})$. Apply the corestriction map $cores_{G_{2i}}:H^{2}(G_{2i},G_{1p_{i}})\rightarrow H^{2}(G_{2},G_{1p_{i}})$. Then by using \cite[Corollary 2.4.9]{Wei69}, we get  $[(\sigma\circ\varepsilon_{1})\varepsilon_{2}^{-1}]^{|G_{2}:G_{2i}|}=1$ for all $1\leq i\leq k$. Hence, the order of $[(\sigma\circ\varepsilon_{1})\varepsilon_{2}^{-1}]$ is coprime with all elements of $\pi(G_{2})$. But, by \cite[Proposition 3.1.6]{Wei69}, we have $H^{2}(G_{2},G_{1})^{|G_{2}|}=1$, which implies that $[(\sigma\circ\varepsilon_{1})\varepsilon_{2}^{-1}]=1$ in $H^{2}(G_{2},G_{1})$, as required.
\end{proof}

\begin{defn}
Let $H$ be a subgroup of $G_{1}$. The perturbed direct products $G_{1}\underset{\varepsilon _{1}}{\times }G_{2}$ and $%
G_{1}\underset{\varepsilon _{2}}{\times }G_{2}$ are called
$(H,G_{2})$-isomorphic if there exists a $G_{2}$-isomorphism $\varphi=\left(
\begin{array}{cc}
\varphi _{11} & \varphi _{12} \\
\varphi _{21} & id_{G_{2}}%
\end{array}%
\right) $  between them such that $\varphi_{11}/H=id_{H}$.
\end{defn}

\begin{prop}
 Let $G_{2}$ be a group which acts trivially on an abelian group $G_{1}$. Let $H=Im(\varepsilon_{1})$. The perturbed direct products $G_{1}\underset{\varepsilon_{1}}{\times }G_{2}$ and $G_{1}\underset{\varepsilon_{2}}{\times }G_{2}$ are $(H,G_{2})$-isomorphic if and only if $\varepsilon_{1}\varepsilon_{2}^{-1}\in B^{2}(G_{2},G_{1})$.
\end{prop}

\begin{proof}
 Indeed, if the perturbed direct products $G_{1}\underset{\varepsilon_{1}}{\times }G_{2}$ and $G_{1}\underset{\varepsilon_{2}}{\times }G_{2}$ are $(H,G_{2})$-isomorphic, then there exists an isomorphism
 $\varphi=\left(
     \begin{array}{cc}
     \varphi _{11} & \varphi _{12} \\
     \varphi _{21} & id_{G_{2}}%
     \end{array}%
     \right) $ such that $\varphi _{21}\in Hom(G_{1},Z(G_{2}))$ and $\varphi_{11}/H=id_{H}$. So, evaluate the left hand side and right hand side of the equality $\varphi (1,y) \underset{\varepsilon_{2}}{\bullet }\varphi ( 1,y')=\varphi (\varepsilon_{1}(y,y'),yy')$, we obtain
     \begin{enumerate}
       \item $\varphi _{21}(\varepsilon_{1}(y,y'))yy'=yy'$,
       \item $\varphi _{11}(\varepsilon_{1}(y,y'))\varphi _{12}(yy')\varepsilon_{2}(\varphi _{21}(\varepsilon_{1}(y,y')),yy')=\varphi _{12}(y) \varphi _{12}(y')\varepsilon_{2}(y,y')$.
     \end{enumerate}
     The first equality implies that $Im(\varepsilon_{1})\leq Ker(\varphi _{21})$. So the second equality gives us $$\varphi _{11}(\varepsilon_{1}(y,y'))\varepsilon_{2}(y,y')^{-1}=\varphi _{12}(y) \varphi _{12}(y')\varphi _{12}(yy')^{-1}.$$ Thus $(\varphi _{11}\circ\varepsilon_{1})\varepsilon_{2}^{-1}\in B^{2}(G_{2},G_{1})$ and therefore $\varepsilon_{1}\varepsilon_{2}^{-1}\in B^{2}(G_{2},G_{1})$ since $\varphi_{11}/H=id_{H}$. The proof of the converse is clear and similar to the proof of the converse of Proposition \ref{centerless} and then it is omitted.
\end{proof}

Using the preceding proposition, the proof of the following result is similar to the proof of Theorem \ref{centerthm} and then we omit the details.

\begin{thm}
 Let $G_{2}$ be a finite group which acts trivially on an abelian torsion group $G_{1}$. Let $H=Im(\varepsilon_{1})$ and $H_{i}=Im(\varepsilon_{1i})$ for all $1\leq i\leq k$. The perturbed direct products $G_{1}\underset{\varepsilon_{1}}{\times }G_{2}$ and $G_{1}\underset{\varepsilon_{2}}{\times}G_{2}$ are $(H,G_{2})$-isomorphic if and only if $G_{1p_{i}}\underset{\varepsilon_{1i}}{\times }G_{2i}$ and $G_{1p_{i}}\underset{\varepsilon_{2i}}{\times }G_{2i}$ are  $(H_{i},G_{2i})$-isomorphic for all $1\leq i\leq k$.
\end{thm}


\section{Isomorphisms leaving the kernel group invariant}


 We need the following definition.

\begin{defn}
Let $\varepsilon_{1}$, $\varepsilon_{2} \in
Z^{2}(G_{2},G_{1})$. The perturbed direct products $G_{1}\underset{\varepsilon_{1}}{\times }G_{2}$ and $G_{1}\underset{\varepsilon_{2}}{\times }G_{2}$ are called upper isomorphic if there exists an isomorphism $\varphi :G_{1}\underset{\varepsilon_{1}}{\times }G_{2}\longrightarrow
G_{1}\underset{\varepsilon_{2}}{\times }G_{2}$ leaving $G_{1}$ invariant.
\end{defn}

 The following lemma can be viewed as an immediate consequence of \cite[Theorem 3.7]{Sn20}.
 
\begin{lem}
\label{Main}The perturbed direct products $G_{1}\underset{\varepsilon_{1}}{\times }G_{2}$ and $G_{1}\underset{\varepsilon_{2}}{\times }G_{2}$ are isomorphic by an isomorphism $\varphi=\left(
\begin{array}{cc}
\varphi _{11} & \varphi _{12} \\
1 & \varphi _{22}%
\end{array}%
\right) $
if and only if $\varphi _{11} \in Aut(G_{1})$ and $\varphi _{22} \in Aut(G_{2})$ such that
$$(\varphi _{11}\circ\varepsilon_{1})(\varepsilon_{2}^{-1}\circ(\varphi _{22} \times \varphi _{22} )) \in B^{2}( G_{2},G_{1}).$$
\end{lem}

With the notations and assumptions that preceded Theorem \ref{centerthm}, we have the following.

\begin{prop}
Let $G_{2}$ be a cyclic group which acts trivially on an abelian torsion group $G_{1}$. The perturbed direct products $G_{1}\underset{\varepsilon_{1}}{\times }G_{2}$ and $G_{1}\underset{\varepsilon_{2}}{\times }G_{2}$ are upper isomorphic if and only if $G_{1p_{i}}\underset{\varepsilon_{1i}}{\times }G_{2i}$ and $G_{1p_{i}}\underset{\varepsilon_{2i}}{\times}G_{2i}$ are upper isomorphic for all $1\leq i\leq k$.
\end{prop}

\begin{proof}Indeed, by \cite[Proposition 3.11]{Sn20}, the groups $G_{1}\underset{\varepsilon_{1}}{\times }G_{2}$ and $G_{1}\underset{\varepsilon_{2}}{\times }G_{2}$ are upper isomorphic if and only if there exists $\sigma \in Aut(G_{1})$ such that
$(\sigma\circ\varepsilon_{1})\varepsilon_{2}^{-1} \in B^{2}( G_{2},G_{1})$. Thus, using the same arguments as those used in the proof
of Theorem \ref{centerthm}, we get the required result.
\end{proof}

If $G$ is a group and $g\in G$, we will write $\gamma_{g}$ for the inner automorphism determined by $g$, i.e. $\gamma_{g}$ maps an element $x$ to $gxg^{-1}$.

\begin{thm}
  Let $G_{2}$ be a group which acts trivially on an abelian torsion group $G_{1}$. If the perturbed direct products $G_{1}\underset{\varepsilon_{1}}{\times }G_{2}$ and  $G_{1}\underset{\varepsilon_{2}}{\times }G_{2}$ are upper isomorphic, then $G_{1p_{i}}\underset{\varepsilon_{1i}}{\times }G_{2i}$ and $G_{1p_{i}}\underset{\varepsilon_{2i}}{\times}G_{2i}$ are upper isomorphic for all $1\leq i\leq k$.
\end{thm}

\begin{proof}Indeed, if the groups $G_{1}\underset{\varepsilon_{1}}{\times }G_{2}$ and $G_{1}\underset{\varepsilon_{2}}{\times }G_{2}$ are upper isomorphic, then by Lemma \ref{Main}, there exists an isomorphism  $\varphi=\left(
\begin{array}{cc}
\varphi _{11} & \varphi _{12} \\
1 & \varphi _{22}%
\end{array}%
\right) $ between them  such that $\varphi _{11} \in Aut(G_{1})$ and $\varphi _{22} \in Aut(G_{2})$. Since $G_{2i}\in Syl_{p_{i}}(G_{2})$, it follows that $\varphi _{22}(G_{2i})=g_{i}G_{2i}g_{i}^{-1}$ for some $g_{i}\in G_{2}$ and then $(\gamma_{g_{i}^{-1}}\circ\varphi_{22})(G_{2i})=G_{2i}$. Therefore $\rho_{i}=res_{G_{2i}}(\gamma_{g_{i}^{-1}}\circ\varphi _{22})\in Aut(G_{2i})$. Further, we have $\sigma_{i}=res_{G_{1p_{i}}}(\varphi_{11})\in Aut(G_{1p_{i}})$ since $G_{1}$ is torsion. Now, by a simple calculation, we get $\gamma_{t_{2}(g_{i})^{-1}}\circ\varphi=\left(
\begin{array}{cc}
\varphi _{11} & \varphi^{i}_{12} \\
1 & \varphi^{i}_{22}%
\end{array}%
\right) $, where $\varphi^{i}_{22}=\gamma_{g_{i}^{-1}}\circ\varphi_{22}$ and  $\varphi^{i}_{12}(y)=\varphi_{12}(y)\varepsilon_{2}(g_{i}^{-1},g_{i})^{-1}\varepsilon_{2}(g_{i}^{-1},\varphi_{22}(y))\varepsilon_{2}(g_{i}^{-1}\varphi_{22}(y),g_{i})$ for all $y\in G_{2}$. Hence, the groups $G_{1}\underset{\varepsilon_{1}}{\times }G_{2}$ and $G_{1}\underset{\varepsilon_{2}}{\times }G_{2}$ are lower isomorphic by the isomorphism $\varphi'=\gamma_{t_{2}(g_{i})^{-1}}\circ\varphi$. Therefore, by Lemma \ref{Main}, we have $(\varphi _{11}\circ\varepsilon_{1})(\varepsilon_{2}^{-1}\circ(\varphi^{i}_{22} \times \varphi^{i}_{22})) \in B^{2}(G_{2},G_{1})$. Thus, by applying the restriction map $res_{G_{2i}}$, we get $(\sigma_{i}\circ\varepsilon_{1i})(\varepsilon_{2i}^{-1}\circ(\rho_{i}\times \rho_{i}))\in B^{2}(G_{2i},G_{1p_{i}})$. This completes the proof.
\end{proof}

In particular, if $G_{2}$ is a finite nilpotent group, then the converse of the preceding result holds, as shown in the following proposition.

\begin{prop}\label{nil1}
  Let $G_{2}$ be a finite nilpotent group which acts trivially on an abelian torsion group $G_{1}$. Suppose that $G_{1p_{i}}\underset{\varepsilon_{1i}}{\times }G_{2i}$ and   $G_{1p_{i}}\underset{\varepsilon_{2i}}{\times }G_{2i}$ are upper isomorphic for all $1\leq i\leq k$. Then, the perturbed direct products $G_{1}\underset{\varepsilon_{1}}{\times }G_{2}$ and   $G_{1}\underset{\varepsilon_{2}}{\times }G_{2}$ are upper isomorphic.
\end{prop}

\begin{proof}
Indeed, if $G_{1p_{i}}\underset{\varepsilon_{1i}}{\times }G_{2i}$ and $G_{1p_{i}}\underset{\varepsilon_{2i}}{\times }G_{2i}$ are upper isomorphic, then by Lemma \ref{Main}, there exist $\sigma_{i} \in Aut(G_{1p_{i}})$, $\rho_{i} \in Aut(G_{2i})$ such that $(\sigma_{i}\circ\varepsilon_{1i})(\varepsilon_{2i}^{-1}\circ(\rho_{i}\times \rho_{i}))\in B^{2}(G_{2i},G_{1p_{i}})$. Since $G_{2}$ is nilpotent, it is the direct product of its Sylow subgroups and then $\rho=(\rho_{i})_{1\leq i\leq k}\in Aut(G_{2})$. Furthermore, we have $\sigma=(\sigma_{i})_{1\leq i\leq k}\in Aut(G_{1})$. Thus, we have $res_{G_{2i}}[(\sigma\circ\varepsilon_{1})(\varepsilon_{2}^{-1}\circ(\rho\times \rho))]=1$ in $H^{2}(G_{2i},G_{1p_{i}})$. The rest of the proof is similar to the second part of the proof of Theorem \ref{centerthm}, and so is omitted.
\end{proof}


\section{Isomorphisms inducing a commuting automorphism on the quotient group}


Let $G$ be a group. An automorphism $\rho$ of $G$ is called a commuting automorphism of $G$ if for each
$x\in G$, $\rho(x)$ commutes with $x$. The set of all commuting automorphisms of $G$ is denoted by $\mathcal{A}(G)$. The group $Aut_{c}(G)=C_{Aut(G)}(G/Z(G))$  of central automorphisms of $G$ is always a subset of $\mathcal{A}(G)$.

\begin{defn}
The perturbed direct products $G_{1}\underset{\varepsilon _{1}}{\times }G_{2}$ and $%
G_{1}\underset{\varepsilon _{2}}{\times }G_{2}$ are called
$\mathcal{A}$-isomorphic if there exists an isomorphism $\varphi=\left(
\begin{array}{cc}
\varphi _{11} & \varphi _{12} \\
\varphi _{21} & \varphi _{22}%
\end{array}%
\right) $  between them such that $\varphi _{22}\in\mathcal{A}(G_{2})$. In particular, if $\varphi _{22}\in Aut_{c}(G_{2})$, then the groups $G_{1}\underset{\varepsilon _{1}}{\times }G_{2}$ and $%
G_{1}\underset{\varepsilon _{2}}{\times }G_{2}$ are called
$c$-isomorphic.
\end{defn}

\begin{prop} Let $G_{2}$ be a centerless perfect group. Then, the perturbed direct products $G_{1}\underset{\varepsilon_{1}}{\times }G_{2}$ and $G_{1}\underset{\varepsilon_{2}}{\times }G_{2}$ are $\mathcal{A}$-isomorphic if and only if there exists $\sigma \in Aut(G_{1})$ such that $(\sigma\circ\varepsilon_{1})\varepsilon_{2}^{-1}\in B^{2}( G_{2},G_{1})$.
\end{prop}

\begin{proof}
Since $G_{2}$ is a centerless perfect group, it follows that $\mathcal{A}(G_{2})=\{id_{G_{2}}\}$ \cite{La86}. Therefore, the result follows directly from Proposition \ref{centerless}.
\end{proof}

Let $G_{2}$ be a finite group such that $\pi(G_{2})=\{p_{1}, p_{2}, \ldots, p_{k}\}$. Let $G_{2i}$ be a Sylow $p_{i}$-subgroup of $G_{2}$ for each $1\leq i\leq k$. Then, we have the following result.

\begin{prop}
Suppose that all of the sylow subgroups of $G_{2}$ are of maximal class such that $\log_{p_{i}}|G_{2i}|\geq 4$ for all $1\leq i\leq k$. Then, the perturbed direct products $G_{1}\underset{\varepsilon_{1}}{\times }G_{2}$ and $G_{1}\underset{\varepsilon_{2}}{\times }G_{2}$ are upper $\mathcal{A}$-isomorphic if and only if they are upper $c$-isomorphic.
\end{prop}

\begin{proof}
Assume that $G_{1}\underset{\varepsilon_{1}}{\times }G_{2}$ and $G_{1}\underset{\varepsilon_{2}}{\times }G_{2}$ are $\mathcal{A}$-isomorphic by an isomorphism $\varphi=\left(
\begin{array}{cc}
\varphi _{11} & \varphi _{12} \\
1 & \varphi _{22}%
\end{array}%
\right)$. By \cite[Remark 4.2 (ii)]{De02}, each Sylow subgroup of $G_{2}$ is normalized by $\mathcal{A}(G_{2})$. Therefore, we have $Res_{G_{2i}}(\varphi _{22})\in\mathcal{A}(G_{2i})$ for all $1\leq i\leq k$. Using the assumptions and \cite[Theorem 3.4]{FO13}, we get $\mathcal{A}(G_{2i})=Aut_{c}(G_{2i})$ for all $1\leq i\leq k$. Hence, by \cite[Remark 4.3]{De02}, we have  $\varphi_{22}\in Aut_{c}(G_{2})$. Thus, $G_{1}\underset{\varepsilon_{1}}{\times }G_{2}$ and $G_{1}\underset{\varepsilon_{2}}{\times }G_{2}$ are upper $c$-isomorphic. The other direction is clear since $Aut_{c}(G_{2})$ is a subset of $\mathcal{A}(G_{2})$.
\end{proof}

With the notations that preceded Theorem \ref{centerthm}, we obtain the following proposition.

\begin{prop}
Let $G_{2}$ be a finite nilpotent group which acts trivially on an abelian torsion group $G_{1}$. If all the sylow subgroups of $G_{2}$ are of coclass at most two, then the perturbed direct products $G_{1}\underset{\varepsilon_{1}}{\times }G_{2}$ and $G_{1}\underset{\varepsilon_{2}}{\times }G_{2}$ are upper $\mathcal{A}$-isomorphic if and only if $G_{1p_{i}}\underset{\varepsilon_{1i}}{\times }G_{2i}$ and $G_{1p_{i}}\underset{\varepsilon_{2i}}{\times}G_{2i}$ are upper $\mathcal{A}$-isomorphic for all $1\leq i\leq k$.
\end{prop}

\begin{proof}
Indeed, by Lemma \ref{Main}, if the groups $G_{1p_{i}}\underset{\varepsilon_{1i}}{\times }G_{2i}$ and $G_{1p_{i}}\underset{\varepsilon_{2i}}{\times }G_{2i}$ are upper $\mathcal{A}$-isomorphic, then there exist $\sigma \in Aut(G_{1})$, $\rho \in \mathcal{A}(G_{2})$ such that $(\sigma\circ\varepsilon_{1})(\varepsilon_{2}^{-1}\circ(\rho\times \rho))\in B^{2}(G_{2},G_{1})$. Clearly, we have
$\rho_{i}=res_{G_{2i}}(\rho)\in \mathcal{A}(G_{2i})$ and $\sigma_{i}=res_{G_{1p_{i}}}(\sigma)\in Aut(G_{1p_{i}})$ for all $1\leq i\leq k$. Hence, $res_{G_{2i}}((\sigma\circ\varepsilon_{1})(\varepsilon_{2}^{-1}\circ(\rho\times \rho)))=(\sigma_{i}\circ\varepsilon_{1i})(\varepsilon_{2i}^{-1}\circ(\rho_{i}\times \rho_{i}))\in B^{2}(G_{2i},G_{1p_{i}})$. This proves the only if direction. For the converse, suppose that the groups $G_{1p_{i}}\underset{\varepsilon_{1i}}{\times }G_{2i}$ and $G_{1p_{i}}\underset{\varepsilon_{2i}}{\times}G_{2i}$ are upper $\mathcal{A}$-isomorphic for all $1\leq i\leq k$. By Proposition \ref{nil1}, the groups $G_{1}\underset{\varepsilon_{1}}{\times }G_{2}$ and $G_{1}\underset{\varepsilon_{2}}{\times }G_{2}$ are isomorphic by an isomorphism $\varphi=\left(
\begin{array}{cc}
\sigma & \eta \\
1 & \rho%
\end{array}%
\right) $ such that $\rho=(\rho_{i})_{1\leq i\leq k}\in Aut(G_{2})$ where $\rho_{i}\in\mathcal{A}(G_{2i})$. By \cite[Corollary 3.4]{AAV19}, we have $\mathcal{A}(G_{2})\cong\prod_{i=1}^{k}\mathcal{A}(G_{2i})$. Thus $\rho\in \mathcal{A}(G_{2})$ and then $G_{1}\underset{\varepsilon_{1}}{\times }G_{2}$ and $G_{1}\underset{\varepsilon_{2}}{\times }G_{2}$ are upper $\mathcal{A}$-isomorphic.
\end{proof}



\end{document}